\documentclass[11pt]{amsart}
\usepackage{amsmath,amssymb,amsthm,amscd}

\def\p{\partial}
\def\R{\mathbb{R}}

\def\Z{\mathbb{Z}}

\def\l{\lambda}

\def\cE{{\mathcal E}}

\def\cJ{{\mathcal J}}

\def\cR{{\mathcal R}}
\def\cS{{\mathcal S}}

\newtheorem{thm}{Theorem}

\newtheorem{prop}{Proposition}[section]
\newtheorem{lemma}[prop]{Lemma}

\newtheorem{defn}[prop]{Definition}
\newtheorem{conj}[prop]{Conjecture}

\begin{document}
\title{Biharmonic almost complex structures}
\author{Weiyong He}
\email{whe@uoregon.edu}
\address{University of Oregon, Eugene, OR, 97403}

\begin{abstract}We introduce the notion of \emph{biharmonic almost complex structure} on a compact almost Hermitian manifold and we study its regularity and existence in dimension four. First we show that there always exist smooth energy-minimizing biharmonic almost complex structures for any almost Hermitian structure on a compact almost complex four manifold, and all energy-minimizers form a compact set. Then we study the existence problem when the homotopy class of an almost complex structure is specified. We obtain existence of energy-minimizing biharmonic almost complex structures which depends on the topology of $M^4$. When $M$ is simply-connected and non-spin, then for each homotopy class which is uniquely determined by its first Chern class, there exists an energy-minimizing biharmonic almost complex structure. When $M$ is simply-connected and spin, for each first Chern class, there are exactly two homotopy classes corresponding to the first Chern class. Given a homotopy class $[\tau]$ of an almost complex structure, there exists a canonical operation on the homotopy classes $p$ satisfying $p^2=\text{id}$ such that  $p([\tau])$ and $[\tau]$ have the same first Chern class. We prove that  there exists an energy-minimizing biharmonic almost complex structure in (at least) one of the two  homotopy classes, $[\tau]$ and $p([\tau])$.  In general if $M$ is not necessarily simply-connected, we prove that there exists an energy-minimizing biharmonic almost complex structure in (at least) one of the two  homotopy classes $[\tau]$ and $p([\tau])$.

The study of biharmonic almost complex structures should have many applications, in particular for the smooth topology of the underlying almost complex four manifold. We briefly discuss an approach by considering the moduli space of biharmonic almost complex structures and propose a conjecture. 
\end{abstract}

\maketitle

\section{Introduction}

We continue to study almost complex structures on a compact Riemannian (almost Hermitian) manifold $(M, g)$ of dimension $m=4$.   In \cite{He17}, we have studied the energy-minimizing harmonic almost complex structures with respect to the energy functional
\[
\cE_1(J)=\int_M |\nabla_g J|^2dv. 
\]
The main results we have obtained in \cite{He17} are existence and regularity of generalized $W^{1, 2}$ almost complex structures, parallel to the theory of harmonic maps. For example, we showed that there is always an energy-minimizing harmonic almost complex structure and it is smooth away from (Hausdorff) codimension three singularities. In general the regularity result cannot be improved and these objects are not necessarily smooth.  
In this paper we want to study the energy-minimizing and critical almost complex structures with respect to  the energy functional
\begin{equation}
\cE_2(J)=\int_M |\Delta_g J|^2 dv,\; J\in \cJ_g,
\end{equation}
where $\cJ_g$ is the space of almost complex structures which are compatible with $g$.
We call these objects \emph{biharmonic almost complex structures}
since these objects are tensor-valued version of biharmonic maps.

The notion of biharmonic map might first appear in Eells-Lemaire \cite{EL}, where the authors introduced the notion of $k$-harmonic map.  When $k=2$, two-harmonic maps were studied in 1980s by Jiang \cite{JGY}. We refer the readers to the survey paper \cite{MO} for the history of the subject. 
Extrinsic biharmonic maps were introduced and studied by Chang, Wang and Yang in \cite{CWY}, where the authors proved the smooth regularity of $W^{2, 2}$ biharmonic maps from a disc in $\R^4$ to the spheres (partial regularity was also obtained when the disc is in dimension five and up). In general Wang \cite{Wang} proved that weakly $W^{2, 2}$ biharmonic map from a disc in $\R^4$ to any target is smooth. 

Biharmonic almost complex structure is a tensor-valued version of biharmonic maps. 
The structural equations of the biharmonic maps and the biharmonic almost complex structures share many similarities, which play an important role in the proof of the smooth regularity. On the other hand, 
there are substantial differences between biharmonic almost complex structures and biharmonic maps. The prominent difference might be as follows. 
An almost complex structure $J$ with zero energy ($\cE_2(J)=0$) defines a K\"ahler structure $\nabla J=0$. Hence an energy-minimizing almost Hermitian structure $(g, J)$ can be viewed as a generalization of a K\"ahler structure in some sense. The existence of a K\"ahler structure gives a strong restriction of the manifold itself.  As a comparison, zero energy harmonic/biharmonic maps are trivial.

We shall see that  energy-minimizing biharmonic almost complex structures exist in abundance. We try to study these objects to learn the manifold itself, just as the existence of K\"ahler structure gives a lot information of the manifold itself. This is our main motivation to consider energy-minimizing almost complex structures. 
The first main result of the paper is the following,
\begin{thm}\label{main1}A $W^{2, 2}$-biharmonic almost complex structure on $(M^4, g)$ is smooth. Moreover, on any compact almost Hermitian manifold $(M^4, g)$, there always exist minimizers of the energy functional $\cE_2(J)$ which are smooth biharmonic almost complex structures.  All such energy-minimizers form a compact set. 
\end{thm}
The existence of energy-minimizing almost complex structures is a standard practice 
of calculus of variations. The proof of the regularity result is rather delicate and our method was partly inspired by the theory of biharmonic maps, in particular the approach used in Chang-Wang-Yang \cite{CWY}.  Certainly there are substantial differences, partly due to the nature of tensor-valued functions. 
We will present the proof for a much more general system in a companion paper, joint with R. Q. Jiang \cite{HJ19} and the smooth regularity in Theorem \ref{main1} will follow as a special case. 

An important question in the theory of harmonic (and biharmonic) maps is the existence in a given homotopy class.  
Our second main result concerns with the  existence of energy-minimizing almost complex structures in a fixed homotopy class. 
The topology of $M$ will play some important role in the following. Let $A$ denote the set of homotopy classes of almost complex structures on $M$, so the first Chern class gives a map
 \[
 c_1: A\rightarrow H^2(M, \mathbb{Z}).
 \]
 Donaldson \cite[Section 6]{Donaldson90} defined a map $p: A\rightarrow A$ with $p^2=\text{id}$ and $c_1\circ p=c_1$ using the fact that $\pi_4(S^2)=\mathbb{Z}_2$. 
First we assume that $M$ is simply-connected.  If $M$ is non-spin, then $\sigma=p(\sigma)$ (see \cite{P1}) and the first Chern class $c_1$ determines $\sigma$. If $M$ is spin, then the first Chern class $c_1$ corresponds exactly to two homotopy classes of $\sigma$ and $p(\sigma)$.  

\begin{thm}\label{main2} 
Let $(M, g)$ be a compact, simply-connected almost Hermitian four manifold. If $M$ is non-spin, every homotopy class of almost complex structures contains an energy-minimizing biharmonic almost complex structure. If $M$ is spin, at least one of the homotopy classes $\sigma$ and $p(\sigma)$ contains
  an energy minimizing biharmonic almost complex structure.  \end{thm}

When $M$ is not necessarily simply-connected, we have the following,
\begin{thm}\label{main3}
Let $(M, g)$ be a compact almost Hermitian four manifold. Then at least one of the homotopy classes $\sigma$ and $p(\sigma)$ contains  an energy minimizing biharmonic almost complex structure. 
\end{thm}
It should be true that $\sigma=p(\sigma)$ if and only if $M$ is non-spin, regardless whether $M$ is simply-connected or not. Since our knowledge is not sufficient to justify this statement (after consulting some experts in the filed), hence we state our results separately depending on whether $M$ is simply-connected. 
 
We briefly discuss the proof  of Theorem \ref{main3}. 
We might compare our results with the famous result of Sacks-Uhlenbeck \cite{SU1} on the existence of harmonic maps in  a homotopy class from a two sphere $S^2$ to a compact manifold $M$. Sacks and Uhlenbeck  \cite{SU1} considered a perturbation elliptic system for harmonic maps and one technical core is the bubble analysis. In our case the system for biharmonic almost complex structures is fourth order and a natural perturbed biharmonic system becomes much more complicated. We are not able to prove a regularity result for such a system. This obstructs us to adopt Sacks-Uhlenbeck's approach in \cite{SU1} using perturbed functional and perturbed biharmonic system.

Our approach is to analyze an energy-minimizing sequence directly in a fixed homotopy class. We will get a limit in $W^{2, 2}$ but the convergence is weak (it is convergent strongly in $W^{1, 2}$). Using the special structure of the almost complex structure, we can argue that the weak limit satisfies the elliptic system weakly for biharmonic almost complex structures. By the regularity, we know the limit is a smooth biharmonic almost complex structure, which might not be in the same homotopy class due to the weak convergence. Hence the main difficulty, as in other similar situations, is to understand what exactly happens if the strong convergence fails. 

 A major technical result for us is an $\epsilon$-regularity for an energy-minimizing sequence in a fixed homotopy class: we prove that if the convergence of a weakly convergent energy-minimizing subsequence fails to be strong in $W^{2, 2}$, then there must be energy concentration around finitely many isolated points. Unlike the classical $\epsilon$-regularity in  the theory of harmonic maps, we are working directly with an energy-minimizing sequence and hence have no elliptic system to be dealt with. A main technical tool we develop in the paper is an extension theorem for almost complex structures in $W^{2, 2}$. The classical extension theorems developed by Schoen-Uhlenbeck, F.H. Lin and S. Luckhaus are essentially meant to deal with $W^{1, 2}$ and $W^{1, p}$ Sobolev maps. The method does not work for $W^{2, 2}$ Sobolev maps. One main technical difficulty is to prove such an extension theorem for $W^{2, 2}$ almost complex structures. Our method is  general to work for all $W^{k, p}$ Sobolev maps and we shall discuss such an extension and its application elsewhere.

Motivated by the classical method in the subjects, in particular by F. H Lin's study of harmonic maps \cite{Lin99, Lin01}, we analyze the \emph{defect measure} which measures the failure of strong convergence of energy minimizing $W^{2, 2}$ sequence.  The main technical lemma (Lemma \ref{mainlemma02}) is an $\epsilon$-regularity for the defect measure: if the energy in a ball is sufficiently small, then the defect measure in the limit is identically zero. In short, we prove an $\epsilon$-regularity theorem for a minimizing sequence (for defect measures which are associated with the failure of strong convergence).
We should compare our method with those used in the theory of harmonic maps, notably by F. H. Lin \cite{Lin99, Lin01}. 
Our arguments are very different with those used in the theory of harmonic maps (see F. H Lin \cite{Lin01} for example). In the regularity theory of harmonic maps, several extension theorems of $W^{1, 2}$ ($W^{1, p}$) Sobolev maps play a very important role in \cite{SU, Luck, Lin01} etc, where the authors proved extension theorems using the fact that a ball $B_1$ is bi-Lipschitz equivalent to a cube $[0, 1]^n$. We are dealing with $W^{2, 2}$ objects, it seems to be very hard to generalize the arguments in harmonic maps to $W^{2, 2}$ setting ($W^{k, p}$, for $k\geq 2$).  Instead we prove an extension theorem for $W^{2, 2}$ almost complex structures by direct constructions. Our construction involves choices of a cutoff function and a weighted average technique. A key technical ingredient is that two choices are involved in a very delicate way. 

Our method should have applications in a wide variety of energy-minimizing problems. Recently together with Ruiqi Jiang, Longzhi Lin \cite{HJL} we are able to prove some existence theorem for (extrinsic) poly-harmonic maps in critical dimension in a fixed homotopy class. A key technique is to prove an extension theorem for $W^{k, 2}$ Sobolev maps. Our method should also work for intrinsic biharmonic maps but a new technical issue remains. We will discuss these elsewhere. 

In Appendix we will briefly discuss some possible applications. First we introduce the notion of \emph{biharmonic almost symplectic structures} as generalization of an almost K\"ahler structure. The regularity theory is very similar.  These objects might have interesting applications in symplectic geometry. We also propose a conjecture regarding biharmonic almost complex structures, as a possible application in the theory of smooth topology of a compact  almost complex four manifold. \\

{\bf Acknowledgement:} the author is supported in part by an NSF grant, award no. 1611797. The author thanks Ruiqi Jiang for proofreading the paper. 

\numberwithin{equation}{section}
\numberwithin{thm}{section}
\section{Existence of an energy-minimizer}
In this section we prove the existence of an energy-minimizing biharmonic almost complex structure on $(M, g)$, and derive the Euler-Lagrange equation. To work on these problems,
we have defined the Sobolev spaces of almost complex structures in \cite{He17}. We recall the definition.
\begin{defn}Given an almost Hermitian manifold $(M, g)$ with compatible almost complex structures in $\cJ_g$, we define $W^{k, p}(\cJ_g)$ to be the subspace of $W^{k, p}(T^*M\otimes TM)$ consisting those sections $J\in W^{k, p}(T^*M\otimes TM)$ (locally $J$ is a section of $T^*M\otimes TM$ with $W^{k, p}$ coefficients), such that $J$ satisfies the compatible condition almost everywhere,
\begin{equation}\label{constraint}
J^2=-id, g(J\cdot, J\cdot)=g(\cdot, \cdot).
\end{equation}
\end{defn}
 We can also study energy-minimizing problem without assuming that $J$ is compatible with $g$. Many of the discussions would be the same. But there are several obvious advantages to assume the compatibility condition; one of them is that the compatibility condition implies $J\in L^\infty(T^*M\otimes TM)$ automatically.\\

We have the following, 
\begin{thm}\label{thm-existence}Let $(M, g)$ be a compact Riemannian manifold which has compatible almost complex structures, denoted by $\cJ_g$. Then there exists an energy minimizer of $\cE_2(J)$ in $W^{2, 2}(\cJ_g)$, such that it satisfies the Euler-Lagrangian equation in the weak sense,
\begin{equation}\label{el0}
\Delta^2 J=Q(J, \nabla J, \nabla^2 J, \nabla^3 J),
\end{equation} where the lower order term is given by
\begin{equation}\label{q0}
Q=J\Delta J \Delta J+J\nabla_pJ\nabla_p\Delta J+J\nabla_p\Delta J\nabla_pJ+J\Delta(\nabla_pJ\nabla_pJ). 
\end{equation}
Moreover, energy-minimizers form a sequential compact set in $W^{2, 2}$.
\end{thm}
\begin{proof}We only prove the existence of $W^{2, 2}$ energy minimizer of $\cE_2(J)$ over $W^{2, 2}(\cJ_g).$ We will derive the Euler-Lagrange equation later. 
This is a standard practice of calculus of variations.  Take an energy-minimizing sequence $J_k\in W^{2, 2}(\cJ_g)$, such that
\[
\cE_2(J_k)\rightarrow \inf_{J\in W^{2, 2}(\cJ_g)} \cE_2(J)
\]
Since $J_k\in W^{2, 2}(TM\otimes T^*M)$, it follows from Kondrachov compactness that  a subsequence, still denoted by  $J_k$, converges strongly in $W^{1, 2}(T^*M\otimes TM)$ and weakly in $W^{2, 2}(T^*M\otimes TM)$. Denote the limit by $J_0\in W^{2, 2}(TM\otimes T^*M)$. The strong convergence in $W^{1, 2}$ implies in particular that $J_k$ converges to $J_0$ almost everywhere, hence $J_0$ satisfies the compatible condition \eqref{constraint} almost everywhere. Hence $J_0\in W^{2, 2}(\cJ_g)$. The weak convergence in $W^{2, 2}$ implies that
\[
\cE_2(J_0)\leq \liminf \cE_2(J_k)=\inf_{J\in W^{2, 2}(\cJ_g)} \cE_2(J).
\]
This forces that $\cE_2(J_0)=\min  \cE_2(J)$ and $J_0$ is an energy-minimizer. Moreover this also implies that the convergence $J_k\rightarrow J_0$ is strongly in $W^{2, 2}$ and implies that the set of all energy-minimizers is compact.
\end{proof}

In the next we derive the Euler-Lagrange equation for $\cE_2(J)$. First we assume $J\in \cJ_g$.  Denote the variation of $J$ by $S$, then $S$ satisfies the constraints
\begin{equation}\label{constraint1}
JS+SJ=0, g(S\cdot,\cdot)+g(\cdot, S\cdot)=0
\end{equation}
Note that $\tilde S=JS$ also satisfies \eqref{constraint1} if $S$ does. In other words, $\cJ_g$ inherits a canonical almost complex structure. 
We compute the first variation, 
\[
\delta \cE_2(J)=2\int_M (\Delta^2 J, S)dv=0
\]
It follows that
\[
\int_M (\Delta^2 J, S)dv=\int_M (J(\Delta^2 J) J, S)dv=0.
\]
Hence we have
\[
\int_M (\Delta^2 J+J(\Delta^2 J) J, S)dv=0.
\]
Note that $\Delta^2 J+J(\Delta^2 J) J$ satisfies the constraint \eqref{constraint1}, hence we get the Euler-Lagrange equation
\begin{equation}\label{biharmonic1}
\Delta^2 J+J(\Delta^2 J) J=0.
\end{equation}
An equivalent form is $[\Delta^2 J, J]=0$. Since $J^2=-id$, we can rewrite the equation as (applying $\Delta^2 (J^2)=0$),
\begin{equation}\label{el}
\Delta^2 J=Q(J, \nabla J, \nabla^2 J, \nabla^3 J),
\end{equation}
where $Q$ is given by \eqref{q0}. 
Now we derive the Euler-Lagrange equation for critical points in $W^{2, 2}(\cJ_g)$. 
\begin{prop}\label{weaksol}A critical point of $\cE_2(J)$ on $W^{2, 2}(\cJ_g)$ satisfies the Euler-Lagrange equation \eqref{el} in the following weak sense, for any $T\in W^{2, 2}(T^*M\otimes TM)\cap L^\infty$, we have
\begin{equation}\label{el1}
\int_M (\Delta J-J\nabla J\nabla J, \Delta T)dv+\int_M (A, T)dv+\int_M(B, \nabla T) dv=0,
\end{equation}
where we write
\[
\begin{split}
A=&J\Delta J\Delta J+\nabla_p J\nabla_p J\Delta J-\Delta J\nabla_p J\nabla_p J +\nabla_p J\Delta J\nabla_p J\\
B=&\nabla J \Delta J J+J\Delta J\nabla J
\end{split}
\]
We can also write the equation in an equivalent form,
\begin{equation}\label{el00}
\int_M  (\Delta J, \Delta T)dv+\int_M (A, T)dv+\int_M(\tilde B, \nabla T) dv=0,
\end{equation}
where we write
\[
\tilde B=B+\nabla (J\nabla_pJ\nabla_pJ). 
\]
\end{prop}
\begin{proof}
Assume that $T\in W^{2, 2}(\Gamma(T^*M\otimes TM))\cap L^\infty$, denote $A=T+JTJ$, then $JA+AJ=0$. Denote $S=A-gAg^{-1}$, then $S$ satisfies the constraints \eqref{constraint1}. Hence taking variation we have
$\int_M (\Delta J, \Delta S)dv=0.$
Since $J$ is $g$-compatible, this implies
$\int_M (\Delta J, \Delta A)dv=0$. In other words, we have
\[
\int_M (\Delta J, \Delta T)dv+\int_M (\Delta J, \Delta (JTJ))dv=0. 
\]
Note that if $J$ is smooth, we get that $\Delta^2 J+J(\Delta^2 J)J=0$, which can be written as \eqref{el}.  We need to derive the Euler-Lagrangian equation for $J\in W^{2, 2}$. 
We compute
\[
\Delta (JTJ)=J (\Delta T) J+\Delta J T J+J T\Delta J+2\nabla J \nabla T J+2\nabla J T\nabla J+2J\nabla T\nabla J
\]
Since $J$ is $g$-compatible, we have the following, 
\[
\begin{split}
(\Delta J, J \Delta T J)=&\Delta J^j_i J_a^b\Delta T_b^c J_c^d g_{jd} g^{ai}\\
=& \Delta J^j_i J_a^b J_c^d g_{jd}g^{ai } \Delta T_b^c\\
=&J_a^i\Delta J^j_i J_j^d \Delta T_b^c g^{ab} g_{cd}\\
=&(J\Delta J J, \Delta T)
\end{split}
\]
Similarly, we have
\[
(\Delta J, \Delta J T J)=(\Delta J \Delta J J, T)
\]
Hence we get that
\begin{equation}\label{el2}
\int_M (\Delta J+J\Delta J J, \Delta T)+ (A_0, T)+2(B, \nabla T)=0,
\end{equation}
where we have
\[
\begin{split}
A_0=&\Delta J\Delta J J+J\Delta J\Delta J +2\nabla_p J\Delta J\nabla_p J\\
B=&\nabla J \Delta J J+J\Delta J\nabla J
\end{split}
\]
On the other hand, we use $\Delta (J^2)=0$ to conclude that
\[
\begin{split}
&\Delta J J+J\Delta J+2\nabla_p J\nabla_p J=0\\
&J\Delta J J=\Delta J-2J\nabla_p J\nabla_p J.
\end{split}
\]
We compute that
\[
\Delta J\Delta J J=J\Delta J\Delta J+2\nabla_p J\nabla_p J\Delta J-2\Delta J\nabla_p J\nabla_p J 
\]
Hence we get that
\[
A_0=2\left(J\Delta J\Delta J+\nabla_p J\nabla_p J\Delta J-\Delta J\nabla_p J\nabla_p J +\nabla_pJ\Delta J\nabla_p J\right)
\]
We denote $A_0=2A$, with 
\[
A=J\Delta J\Delta J+\nabla_p J\nabla_p J\Delta J-\Delta J\nabla_p J\nabla_p J +\nabla_p J\Delta J\nabla_p J
\]
Together with \eqref{el2}, we have
\begin{equation}\label{el3}
\int_M \left(\Delta J-J\nabla J\nabla J, \Delta T\right)+(A, T)+(B, \nabla T)=0
\end{equation}
This completes the proof of \eqref{el1}. Certainly we can also write
\[
\int_M (-J\nabla J\nabla J, \Delta T)dv=\int_M \left(\nabla (J\nabla_p J\nabla_p J), \nabla_i T\right)dv
\]
This would lead to the equation
\begin{equation}\label{el4}
\int_M (\Delta J, \Delta T)dv+\int_M (A, T)dv+\int_M (\tilde B, \nabla T)dv=0,
\end{equation}
where we write
\[
\tilde B=B+\nabla (J\nabla_p J\nabla_p J)
\]
This finishes the proof. 
\end{proof}

In the proof of the above result,  it is rather straightforward to see that Proposition \ref{weaksol} is equivalent to \eqref{biharmonic1} for $J\in W^{2, 2}$, if  \eqref{biharmonic1} is understood as a weak solution in the form
that
\begin{equation}\label{weak-101}
\int_M (\Delta J, \Delta T)dv+\int_M (\Delta J, \Delta(JTJ) ) dv=0. 
\end{equation}
For $J\in W^{2, 2}$, the derivatives of the forms $\Delta^2 J$ and $\nabla_p \Delta J$ are understood in a distribution way properly. With this understanding, Proposition \ref{weaksol} is equivalent to the system \eqref{el} (and $[\Delta^2 J, J]=0$). 
An almost complex structure $J\in W^{2, 2}(\cJ_g)$ is called a \emph{weak biharmonic almost complex structure} if it satisfies \eqref{weak-101}, or its equivalent form in Proposition \ref{weaksol}. 
The formulation \eqref{weak-101} has several equivalent formulations. 

\begin{prop}\label{weakac-01}A weak biharmonic almost complex satisfies the following, for any $T\in \Gamma(TM\otimes T^*M)$
\begin{equation}\label{weak-102}
\int_M (\Delta J, \Delta (TJ-JT))dv=0
\end{equation}
Or equivalently, we have the following,
\begin{equation}\label{weak-103}
\int_M (\Delta J, \Delta T J-J\Delta T)dv+2\int_M (\Delta J, \nabla T\nabla J-\nabla J\nabla T)dv=0.
\end{equation}
As a consequence, the weak limit of a sequence of $W^{2, 2}(\cJ_g)$ biharmonic almost complex structure (with bounded $W^{2, 2}$ norm) is still biharmonic. 
\end{prop}

\begin{proof}
To derive \eqref{weak-102}, we take $T=\tilde T J$ in \eqref{weak-101}.  We observe that
\begin{equation}\label{weak-104}
\int_M (\Delta J, (\Delta J) T)dv=\int_M (\Delta J, T \Delta J)dv.
\end{equation}
Note that
\[
(\Delta J, \Delta J T)=\Delta J_{i}^j \Delta J_a^b T_b^c g_{jc} g^{ia}=\Delta J_i^j \Delta J_a^b g_{jc} g^{ia} T^c_b 
\]
Note that $\Delta J^b_a g^{ia}=\Delta (J^b_a g^{ia})=-\Delta (J^i_a g^{ba})=-(\Delta J^i_a) g^{ba}$. Hence we have
\[
(\Delta J, \Delta J T)=\Delta J_i^j \Delta J_a^b g_{jc} g^{ia} T^c_b =-\Delta J_a^i\Delta J_i^j g_{jc} g^{ab} T_b^c=-(\Delta J\Delta J, T). 
\]
Similarly we have 
\[
(\Delta J, T\Delta J)=-(\Delta J\Delta J, T). 
\]
This proves \eqref{weak-104}. Together with \eqref{weak-102}, this proves \eqref{weak-103}. Now suppose a sequence of weak biharmonic almost complex structures $J_k$, with bounded $W^{2, 2}$ norm, converges weakly in $W^{2, 2}$ to $J_0$, and strongly in $W^{1, 2}$ (and $L^2$). In particular $J_k$ converges to $J_0$ almost everywhere. By passing to the limit, then $J_0$ satisfies \eqref{weak-103}. More precisely, for any fixed  smooth $T$, 
we have (using the fact that smooth endomorphisms are dense in $W^{2, 2}$ and weak convergence of $J_k\rightarrow J_0$ in $W^{2, 2}$), 
\[
\int_M (\Delta (J_k-J_0), \Delta T J_0-J_0\Delta T)dv \rightarrow 0.
\]
The strong convergence $J_k\rightarrow J_0$ in $L^2$ implies that
\[
 \int_M (\Delta J_k, \Delta T (J_k-J_0)-(J_k-J_0)\Delta T)dv\rightarrow 0
\]
Together this implies that
\[
\int_M (\Delta J_k, (\Delta T)J_k-J_k\Delta T)dv\rightarrow \int_M (\Delta J_0, (\Delta T)J_0-J_0\Delta T)dv.
\]
Similarly we have 
\[
\int_M (\Delta J_k, \nabla T\nabla J_k-\nabla J_k\nabla T)dv\rightarrow \int_M (\Delta J_0, \nabla T\nabla J_0-\nabla J_0\nabla T)dv
\]
This completes the proof. 
\end{proof}

\begin{proof}[Proof of Theorem 1]As we have shown above, the existence of energy-minimizing biharmonic almost complex structures is a standard practice and the energy-minimizers form a compact set, which works for all dimensions. To finish the proof of Theorem \ref{thm-existence} we need to show that in dimension four, a weak biharmonic almost complex structure is always smooth. 

Our method is motivated by the theory of biharmonic maps considered in \cite{CWY} in principle. Our method works for  very general semilinear elliptic systems with critical nonlinear growth. We will  refer the smooth regularity to 
 \cite{HJ19}, where together with Ruiqi Jiang we will prove a  general regularity result for elliptic systems. The results in \cite{HJ19} include the elliptic system of (weakly) biharmonic almost complex structures on a compact four manifold $M^4$ as a special case. 

\end{proof}

\section{Biharmonic almost complex structures in a homotopy class}\label{sh3}
A fundamental problem in the theory of harmonic maps is to find harmonic maps in a fixed homotopy class. Similarly we would like to ask the same question for biharmonic almost complex structures. 
We shall see that the topology of $M$ plays a very important role.
The following famous example about almost complex structures on a $K3$ surface constructed by S. Donaldson \cite{Donaldson90} serves as an important example. We recall relevant discussions.
 A compatible almost complex structure $J$ on an oriented Riemannian four-manifold $M$ can be considered as a section of the associated $SO(4)/U(2)$-bundle over $M$ (the twistor space). 
 Let $A$ denote the set of homotopy classes of almost complex structures, so the first Chern class gives a map
 \[
 c_1: A\rightarrow H^2(M, \mathbb{Z}).
 \]
 Donaldson \cite[Section 6]{Donaldson90} defined a map $p: A\rightarrow A$ with $p^2=\text{id}$ and $c_1\circ p=c_1$ as follows. For $\sigma\in A$, $p(\sigma)$ agrees with $\sigma$ outside a small ball in $M$ and over this ball the two compare by the nonzero element of
 \[
 [S^4, SO(4)/U(2)]=[S^4, S^2]\cong \mathbb{Z}/2.
 \]
Donaldson proved \cite[Corollary 6.5]{Donaldson90} that if $\sigma$ is the homotopy class of the standard integrable complex structures, then the homotopy class $p(\sigma)$ of almost complex structures with $c_1=0$ on the $K3$ surface does not contain any integrable representative.
We need the following general result about $\sigma$ and  $p(\sigma)\in A$, which we have learned from P. Teichner through mathoverflow \cite{P1},

\begin{lemma}\label{homotopy1}
If $M$ is simply-connected and is not spin, then the first Chern class uniquely determines the homotopy class of an almost complex structure. When $M$ is simply-connected and spin, there are exactly a pair of homotopy classes $\sigma, p(\sigma)$ corresponding to a given first Chern class. 
\end{lemma}
\begin{proof}This can be proved by obstruction theory for the fibration 
\[
F=SO(4)/U(2)\rightarrow BU(2)\rightarrow BSO(4)
\]
where the fiber $F=S^2$. Start with the tangent bundle of an oriented compact four-manifold $M$. Asking for existence and uniqueness respectively for an almost complex structure on $M$ is equivalent to 
asking for existence and uniqueness respectively of a lift of its Gauss-map $M\rightarrow BSO(4)$ to $M\rightarrow BU(2)$. The obstructions for existence lie in $H^{i+1}(M; \pi_i(F))$ and the obstructions for uniqueness lie in $H^i(M; \pi_i(F))$. Assume that M is simply-connected and let $X$ be its 2-skeleton. The same obstruction theory shows that the tangent bundle of $X$ has a complex structure and it is uniquely determined by the first Chern-class. If such a structure extends to $M$, there is a single uniqueness obstruction in 
\begin{equation}\label{top01}
H^4(M, X; \pi_4(F))=\pi_4(S^2)=\mathbb{Z}/2.
\end{equation}
By naturality, this obstruction is realized by a second almost complex structure on $M$ with the same first Chern class if and only if the quotient map $M\rightarrow M/X=S^4$ induces a nontrivial map on cohomotopy:
\begin{equation}\label{top02}
\pi_4(S^2)=[S^4, S^2]\rightarrow [M, S^2].
\end{equation}
This map is nontrivial if and only if $M$ is spin, see \cite{KMT}. Hence the first Chern-class characterizes almost complex structures on $M$ if and only if $M$ is not spin ($M$ is simply-connected). 
\end{proof}

When $M$ is not necessarily simply-connected, we still consider the map $p: A\rightarrow A$.
Fix a homotopy class $\sigma\in A$,  and denote
$E_0(\sigma)=\inf_{J\in \sigma} \cE_2(J). $
The main result in this section is the following,

\begin{thm}Let $M$ be a compact four manifold with an almost complex structure. Denote $\sigma$ to be one homotopy class of an almost complex structure. Then either $\sigma$ or $p(\sigma)$ contains an energy-minimizing biharmonic almost complex structures. 
\end{thm}

The proof of the above theorem consists of several steps. 
Let $J_k\in \sigma$ be a minimizing sequence such that $\cE_2(J_k)\rightarrow E_0$. A direct computation by integration by parts implies that $J_k$ has uniformly bounded $W^{2, 2}$ norm. Hence by subsequence we can assume that $J_k\rightharpoonup J_0$ weakly in $W^{2, 2}$ and strongly in $W^{1, 2}$. In particular $J_k$ converges to $J_0$ almost everywhere and hence $J_0\in W^{2, 2}(\cJ_g)$. First we have the following,

\begin{lemma}The limit $J_0\in W^{2, 2}(\cJ_g)$ is a weak biharmonic almost complex structure satisfying \eqref{weak-103}. In particular it is smooth.
\end{lemma}
\begin{proof}Let $J_k$ be a minimizing sequence and $J_0$ be its weak limit in $W^{2, 2}(\cJ_g)$. For any fixed smooth $T\in \Gamma(TM\otimes T^*M)$, consider 
\[
S_t= t (T+J_k TJ_k)-tg(T+J_kTJ_k)g^{-1}
\]
Hence $S_t$ satisfies $J_kS_t+S_tJ_k=0$, and $g(S_t, \cdot)+g(\cdot, S_t)=0$. Note that for $|t|$ sufficiently small, $S_t$ has small $L^\infty$ norm and we construct 
\[
J_k^t=J_k (id+S_t)(id-S_t)^{-1}=(id-S_t) J_k (id-S_t)^{-1}. 
\]
We compute, for $|t|$ sufficiently small, 
\begin{equation}\label{app-19}
E(J_k^t)=\int_M |\Delta J_k^t|^2 dv=E(J_k)+4t\int_M (\Delta J_k, \Delta (J_kT-TJ_k))dv+O(t^2),
\end{equation}
where the term $O(t^2)$ denotes the terms of higher order in $t$. We have $|O(t^2)|\leq Ct^2$ for a uniformly bounded constant $C$ (assuming $t$ is small).
A quick way to see \eqref{app-19} is to write the (matrix) expansion of $J_k^t$ as
\[
\begin{split}
J_k^t&=(id-S_t) J_k (id-S_t)^{-1}=(id-S_t) J_k \left(\sum_{l=0}^\infty S_t^l\right)\\
&=J_k+2J_kS_t+2J_k\sum_{l\geq 2} S_t^l
\end{split}
\]
Since $J_k$ is a minimizing sequence, hence when $k\rightarrow \infty$, we have 
\[
\liminf_{k\rightarrow \infty} E(J_k^t)-E(J_k) \geq 0
\]
Hence for $|t|$ sufficiently small,
\[
\limsup_k 4t\int_M (\Delta J_k, \Delta (J_kT-TJ_k))dv\geq \liminf_k 4t\int_M (\Delta J_k, \Delta (J_kT-TJ_k))dv\geq 0. 
\]
In particular, this implies that
\[
\lim\sup_k(\Delta J_k, \Delta (J_kT-TJ_k))dv=\liminf_k \int_M (\Delta J_k, \Delta (J_kT-TJ_k))dv=0.
\]
Similar as in \eqref{weak-103} and \eqref{weak-104}, we use the fact that
\[
\int_M (\Delta J_k, \Delta J_k T)dv=\int_M(\Delta J_k, T\Delta J_k)dv
\]
to conclude that
\[
\lim_k \int_M (\Delta J_k, J_k\Delta T-\Delta T J_k)dv+2\int_M (\Delta J_k, \nabla J_k \nabla T-\nabla T\nabla J_k)dv=0
\]
Since $J_k$ converges to $J_0$ weakly in $W^{2, 2}$ and strongly in $W^{1, 2}$, the above implies that
\[
\int_M (\Delta J_0, J_0\Delta T-\Delta T J_0)dv+2\int_M (\Delta J_0, \nabla J_0 \nabla T-\nabla T\nabla J_0)dv=0
\]
By Proposition \ref{weakac-01}, $J_0$ is a weak biharmonic almost complex structure and hence it is smooth by regularity results in \cite{HJ19}. 
\end{proof}

This minimizing process produces a smooth biharmonic almost complex structures $J_0$.  An important question is to understand whether $J_0$ remains in the homotopy class $\sigma$. Similar as in the theory of harmonic maps, this might not be true since the weak convergence $J_k\rightharpoonup J_0$ does not preserve homotopy classes. To overcome this difficulty, we need certain topology preparations. First we want to understand the first Chern classes of $J_k$ and $J_0$. We recall the following, 

\begin{lemma}[C. Wood \cite{Wood95}]\label{wood-01}Given a compact almost Hermitian manifold $(M, g, J, \omega)$, then the first Chern class can be represented by the following 2-form $\gamma$ such that
\begin{equation}\label{chernclass}
2\pi \gamma=\cR(\omega)+\chi,
\end{equation}
where $\cR(\omega)$ is the curvature operator acting on the K\"ahler form $\omega$, and the 2-form $\chi$ is given by \[\chi(X, Y):=\frac{1}{4}\omega(\nabla _X J, \nabla_Y J).\]
\end{lemma}
Lemma \ref{wood-01} has the following important consequence
\begin{lemma}\label{1stchern-01}The first Chern class $c_1(J_0)=c_1(J_k)$ for $J_k\in \sigma$. 
\end{lemma}

\begin{proof}The first Chern class $c_1(J)$ is the deformation invariant and hence it remains the same in the homotopy class $\sigma$. By Lemma \ref{wood-01}, the first Chern class is represented by the 2-form $\gamma$. We write $\gamma_k$ for $J_k$ and $\gamma_0$ for $J_0$. Since $J_k$ converges strongly to $J_0$ in $W^{1, 2}$, then for any smooth two form $\zeta$, \eqref{chernclass} implies that
\[
\lim_{k\rightarrow \infty}\int_M \gamma_k\wedge \zeta=\int_M \gamma_0\wedge \zeta
\]
It follows that $[\gamma_k]=[\gamma_0]\in H^2(M, \mathbb{Z})$. 
\end{proof}

Hence the weak limit $J_0$ has the same first Chern class as the homotopy class $\sigma$. 
If $M$ is  simply-connected and  non-spin, Lemma \ref{homotopy1} implies that $J_0$ is in the homotopy class $\sigma$. Hence $E(J_0)=\inf E(J_k)$ and a standard variational argument implies that $J_k$ converges to $J_0$ strongly in $W^{2, 2}$ and $J_0$ is the energy minimizer of $E(J)$ in the homotopy class of $\sigma$. Hence we have the following,

\begin{lemma}\label{nonspin}When $M$ is simply-connected and non-spin, then every homotopy class of almost complex structures contains an energy-minimizing biharmonic almost complex structure.
\end{lemma}

If $M$ is simply-connected and spin, the situation is more complicated. Suppose $\sigma$ has the first Chern class $c$ and we use $p(\sigma)$ to denote the other homotopy class which has the same first Chern class $c$. 
By Lemma \ref{homotopy1}, either $J_0$ is in the homotopy class of $\sigma$, or $J_0$ is in the homotopy class of $p(\sigma)$ since $c(J_0)=c$.
In the former case, $J_0$ is an energy-minimizing biharmonic almost complex structure in $\sigma$; in the later case, we prove that there is an energy-minimizing biharmonic almost complex structure in the homotopy class $p(\sigma)$.  
To achieve this, we consider the homotopy class $\sigma$ and $p(\sigma)$ simultaneously.  We consider  $\inf_{J\in \sigma} \cE_2(J)$ and  $\inf_{J\in p(\sigma)}\cE_2(J)$. We assume that $\inf_{J\in \sigma} \cE_2(J)\geq \inf_{J\in p(\sigma)}\cE_2(J)$. A minimizing sequence $J_k$ in $p(\sigma)$ will have a weak limit $J_0$ and $\cE_2(J_0)\geq \inf_{J\in p(\sigma)}\cE_2(J)$, since $J_0$ is either in $\sigma$ or $p(\sigma)$. This will force that $\cE_2(J_0)=\inf_{J\in p(\sigma)}\cE_2(J)$ and hence $J_0$ is an energy minimizing biharmonic almost complex structure. (W can  actually prove that $J_0\in p(\sigma)$ and $J_k$ converges strongly in $W^{2, 2}$ to $J_0$, using the $\epsilon$-regularity results below).

\begin{lemma}\label{spin}When $M$ is simply-connected and spin, then for each first Chern class $c\in H^{2}(M, \mathbb{Z})$ (of an almost complex structure), at least one homotopy class (among two homotopy classes corresponding to $c$) contains an energy-minimizing biharmonic almost complex structures.
\end{lemma}

When $M$ is not necessarily simply-connected, the first Chern class can correspond to infinitely many homotopy classes of almost complex structures. For example, on $S^1\times S^3$, there is only the trivial first Chern class, but there are infinitely many homotopy classes which correspond to the homology group $H_1(S^1\times S^3, \Z)=\Z$. In such cases, we need to derive more precise information for a minimizing sequence $J_k$ which converges weakly to a smooth biharmonic almost complex structure $J_0$. An essential point is that  in our case,  the weak convergence (instead of strong convergence) comes from the energy concentration (around a point). This is a well-known philosophy in the theory of harmonic maps of surfaces. As a consequence, we have the following,

\begin{lemma}Let $J_k$ be an energy-minimizing sequence of $\cE_2(J)$ in the homotopy class $\sigma$. Then the weak limit $J_0$ (of a convergent subsequence) lies in either $\sigma$ or $p(\sigma)$. 
\end{lemma}

If $J_k$ converges to $J_0$ strongly in $W^{2, 2}$, then by B. White \cite{White89}, the homotopy class $J_0$ remains the same. Hence the key is then to analyze what exactly happens when the strong convergence fails. In the following we prove a version of $\epsilon$-regularity for the minimizing sequence $J_k$ and the defect measure. Hence we have that $J_k$ converges strongly in $W^{2, 2}$ to $J_0$ except around finitely many isolated points. This will imply that the homotopy class of $J_0$ is either $\sigma$ or $p(\sigma)$.

Now we are ready to state and prove the $\epsilon$-regularity for the minimizing sequence $J_k$, see Lemma \ref{defect01}. 
Fix a sufficiently small positive number $\epsilon_0$, which depends only on $(M, g)$ and will be specified later. Let $\iota$ be the injectivity radius of $(M, g)$.  Suppose a minimizing sequence $J_k\in \sigma$ converges weakly to $J_0$ in $W^{2, 2}$ and strongly in $W^{1, 2}$ (hence converges almost everywhere). 
For $r\in (0, \iota)$, $p\in M$ and $J\in W^{2, 2}$, denote 
\begin{equation}
\begin{split}&E(r, p)=\int_{B_r(p)}|\Delta J|^2 dv\\
&F(r, p)=\int_{B_r(p)} (|\nabla^2J|^2+|\nabla J|^4)dv
\end{split}
\end{equation}
We write $E_0(r, p), F_0(r, p),  E_k(r, p), F_k(r, p)$ correspondingly for $J_0$ and $J_k$. 
Set $\cS_r=\{p\in M: \liminf_{k\rightarrow \infty} F_k(r, p)\geq \epsilon_0\}.$  Clearly $\cS_r\subset \cS_s$ for $r<s$. 
Denote  \[\cS:=\cap_{r>0} \cS_r=\lim_{r\rightarrow 0}\cS_r.\] 
We introduce  the measures which are all totally bounded,  \begin{equation}\begin{cases}&\mu_k=(|\nabla^2J_k|^2+|\nabla J_k|^4)dv\\
&\mu_0=(|\nabla^2 J_0|^2+|\nabla J_0|^4)dv\\
&\xi_k=|\Delta J_k|^2dv\\
&\xi_0=|\Delta J_0|^2dv
\end{cases}\end{equation} By passing to a subsequence, $\mu_k$ converges weakly to a positive Radon measure $\mu$ and $\xi_k$ converges weakly to a positive Radon measure $\xi$. By Fatou's lemma, there exist  positive Radon measures $\nu$ and $\l$  (called the \emph{defect measure} \cite{Lin99, Lin01}), such that
\[
\begin{cases}
&d\mu=d\nu+d\mu_0\\
 &d\xi=d\l+d\xi_0. 
\end{cases}\] 
Certainly $J_k$ converges to $J_0$ strongly in $W^{2, 2}$ if and only if either $\nu\equiv 0$ or $\l \equiv 0$. 
Hence $\nu$ and $\l$ both measure the failure of strong convergence in $W^{2, 2}$ but they are not necessarily the same in general. The interplay between two defect measures $\nu$ and $\l$ makes our discussions below more complicated than the theory of the harmonic maps, where only the measure $|\nabla u|^2dv$ comes to play. 
We have the following,
\begin{lemma}\label{defect01}The support of  $\nu$ equals $\cS$, which contains at most finitely many points. 
\end{lemma}

\begin{proof}First it is straightforward to see that $\cS$ is contained in the support of $\nu$. Since $J_0$ is smooth, we have for any $x\in M$, $\lim_{r\rightarrow 0}\mu_0(B_r(x))=0.$ If $x\in \cS$, then
\[
\lim_{r\rightarrow 0}\nu(B_r(x))=\lim_{r\rightarrow 0}[\mu(B_r(x))-\mu_0(B_r(x))]\geq \epsilon_0. 
\]
This shows that $\cS$ is contained in the support of $\nu$. 

Next we claim the following, for $r\in (0, r_0]$, where $r_0$ is a fixed, sufficiently small number, if
$\mu(B_{2r}(p))<\epsilon_0$, 
then $\nu\equiv 0$ in $B_{r/2}(p)$. 

We sketch the idea of the proof briefly. If $\mu(B_{2r}(p))<\epsilon_0$ for $\epsilon_0$ sufficiently small, then after passing to a subsequence, we have 
\[
\int_{B_{2r}(p)} (|\nabla^2 J_k|^2+|\nabla J_k|^4)dv<2\epsilon_0
\]
for sufficiently large $k$. This implies that $J_k$ is close to its average in $B_{2r}$ and the same discussion holds for $J_0$. In particular $J_k$ and $J_0$, when restricted in $B_r$, are homotopy to each other. The key is 
to construct a new ``almost minimizing" sequence $\tilde J_k$ in the same homotopy class  such that $\tilde J_k=J_0$ inside $B_{r/2}(p)$ and $\tilde J_k=J_k$ outside $B_r(p)$. 
This would imply $\nu, \l\equiv 0$ in $B_{r/2}(p)$. 

Suppose at the moment the claim is established. If $p$ is in the support of $\nu$, then $\mu(B_{2r}(p))\geq \epsilon_0$ for all sufficiently small $r$, hence $p\in \cS$. Since the total energy is bounded, it follows that $\cS$ contains at most finitely many isolated points. We complete the proof given Lemma \ref{mainlemma02} below, where we establish the claim. 
\end{proof}

First we specify the choice of $r_0$. We can do the scaling $g_r=r^{-2}g$ for $r\leq r_0$. We choose $r_0$  sufficiently small, such that  $g_r$ is sufficiently close to the Euclidean metric in the ball $B_2(p)$ (we identify $B_2(p)$ with the Euclidean ball $B_2$ equipped with the metric $g_r$) such that, 
\begin{equation}\label{metric1}
|g_{ij}-\delta_{ij}|+\sum_{k=1}^4|D^k g_{ij}|<\delta_0,\end{equation}
 where $\delta_0$ measures how close the metric $g_r$ is with respect to the Euclidean metric in $B_2$. We also assume that $g_{ij}(0)=\delta_{ij}, \p g_{ij}(0)=0$. 
 
Note that we choose $r_0, \delta_0$ and $\epsilon_0$ uniformly for any point $p\in M$. The constants $r_0, \delta_0$ and $\epsilon_0$ are all fixed. 
Since $r$ will also be fixed and the energy functionals are scaling invariant, we consider $(M, g_r)$ instead of $(M, g)$.  In other words, we can assume that, by scaling if necessary, $g$ satisfies \eqref{metric1} in any geodesic ball $B_2(p)\subset M$ such that the injectivity radius of $(M, g)$ is bigger than $2$. We identify $B_2(p)$ with Euclidean ball $B_2$ with the metric $g$, where we use Euclidean coordinates. After this choice of the scaling,
  we establish the main technical lemma in this section. 

\begin{lemma}\label{mainlemma02} $\mu(B_{2}(p))\leq\epsilon_0$, 
then $\nu\equiv 0$ in $B_{1}(p)$.
\end{lemma}

\begin{proof}
The proof involves very technical and complicated constructions of a new sequence of ``almost energy-minimizing" almost complex structures $\tilde J_k$ in $\sigma$ such that
\begin{equation}\label{extension01}
\tilde J_k(x)=\begin{cases}
&J_k(x), x\in M\setminus B_1(p)\\
&J_0(x), x\in B_{1-j^{-1}}(p),
\end{cases}
\end{equation}
where $k\geq k_j=k(j)$ is sufficiently large depending on $j$.  In the end we will let $j\rightarrow \infty$ (and $k_j\rightarrow \infty$ accordingly) to get an almost energy-minimizing subsequence. 
Such a construction is a type of extension of an almost complex structure which equals $J_k$ in $M\setminus B_1(p)$ and which equals $J_0$ in $B_{1-j^{-1}}(p)$. The construction happens in a small annulus $B_1(p)\setminus B_{1-j^{-1}}(p)$. The small energy condition $\mu(B_{2}(p))\leq\epsilon_0$ clearly plays a very important role. In particular this implies that $\tilde J_k$ is still in $\sigma$, using a theorem of B. White \cite{White89}. 

The construction involves several delicate choices of small constants and cutoff functions. We shall first briefly explain the process of construction, leaving details to be proved below. We work on $(B_2, g)$. Let $\psi_j: [0, \infty)\rightarrow [0, 1]$ be a smooth cutoff function depending on $j$ such that  
\[
\psi_j(s)=
\begin{cases}
&1, s\geq 1\\
&0, s\leq 1-j^{-1}
\end{cases}
\]
with bounds $|\psi^{'}_j|\leq 3j, |\psi^{''}_j|\leq 10 j^2$. For $x\in B_2$, we denote
\[
J_{k,\, j}(x)=J_k(x)+(J_0(x)-J_k(x))(1-\psi_j(|x|)).
\]
Note that 
\begin{equation}\label{local001}
J_{k,\, j}(x)=\begin{cases}&J_k(x), |x|\geq 1\\
 &J_0(x), |x|\leq 1-j^{-1}. 
\end{cases}
\end{equation}
We extend $J_{k,\, j}$ to $M$ such that it equals $J_k$ outside $B_1(p)$. 
Note that  $J_{k,\, j}$ might not even be invertible for some points $|x|\in (1-j^{-1}, 1)$ since the convergence of $J_k\rightarrow J_0$ does not imply the convergence in $L^\infty$.  To overcome this difficulty we construct a smooth approximation of $J_{k,\, j}$ using a local average technique (a modification of \emph{mollifier}). 

Let $\phi(x)=\phi(|x|)$ be a nonnegative smooth radial cutoff function which is supported in $B=B_1$ with $\int_B \phi dx=1$. 
For any given $J$ and $\rho>0$, we denote
\begin{equation}\label{localaverage1}
J_\rho(x)=\int_{B_\rho (x)}\phi_\rho(y-x) J(y)dy=\int_{B} \phi(z)J(x+\rho z)dz,
\end{equation}
where we us the notation,
\[
\phi_\rho(x)=\rho^{-4} \phi\left(\frac{x}{\rho}\right). 
\]
When $\rho=0$, $\phi_\rho$ is the delta-function and  $J_\rho(x)=J(x)$ (this is also clear from the second equality in \eqref{localaverage1}). 
Certainly, $J_\rho$ is the smooth approximation of $J$ and $J_\rho$ converges to $J$ (in certain norm depending on the regularity of $J$) when $\rho\rightarrow 0$. However, such a smooth approximation does not preserve \eqref{local001} in general. 
Hence we allow $\rho$ to be dependent of $|x|$ and we write $\rho: [0, 2]\rightarrow [0, 1]$ and denote $\rho(x)=\rho(|x|)$.
Such a technique is a modification of the technique used in the theory of harmonic maps by Schoen-Uhlenbeck \cite{SU1}. 
The support of $\rho$ is contained in $(1-j^{-1}, 1)$; in other words, 
\begin{equation}
\rho(|x|)=0, |x| \in [0, 1-j^{-1}]\cup [1, 2].
\end{equation} 
We choose $\rho(1-(2j)^{-1})=\max \rho=\bar \rho$, 
where $\bar \rho$ is a small positive number and it can be chosen such that $10\bar\rho j^2=1$
We require the derivatives of $\rho$ satisfying \begin{equation}\label{constant001}|\rho^{'}|+|\rho^{''}|\leq 10\bar\rho j^2\leq 1.\end{equation}The choice of the function $\rho$ depends crucially on the cutoff function $\psi_j$ and this is the first key point in our construction. 

Given such a function $\rho$, we construct
\begin{equation}
J_{k,\, j,\,\rho(x)}(x)=\int_{B_{\rho(x)}(x)}\phi_\rho(y-x)J_{k,\, j}(y)dy=\int_B \phi(z) J_{k,\, j}(x+\rho(x)z)dz. 
\end{equation}
Similarly we denote
\[
\begin{split}
&J_{0,\, \rho(x)}(x)=\int_{B_{\rho(x)}(x)}\phi_\rho(y-x)J_0(y)dy=\int_B \phi(z)J_0(x+\rho(x)z)dz.\\
&J_{k,\,\rho(x)}(x)=\int_{B_{\rho(x)}(x)}\phi_\rho(y-x)J_k(y)dy=\int_B \phi(z)J_k(x+\rho(x)z)dz.
\end{split}
\]
Note that $J_{0,\,\rho(x)}(x)$, $J_{k, \, \rho(x)}(x)$ and $J_{k,\, j,\,\rho(x)}(x)$ are neither almost complex structures nor compatible with the metric $g$ in general. It is straightforward to see that $J_{0,\,\rho(x)}(x)$ is close to $J_0(x)$ given $\bar\rho$ sufficiently small  ($J_0$ is smooth and $J_{0, \, \rho}$ is a standard approximation). We will also show that $J_{k, \,\rho(x)}(x)$ is close to an almost complex structure and it is almost compatible with the metric $g$, using the Poincare inequality and small energy assumption, where similar arguments were used in  \cite{SU1} for harmonic maps and in \cite{He17} for harmonic almost complex structures. 
One difficulty is to prove that $J_{k,\, j, \,\rho}$ is also close to an almost complex structure, with a suitable choice of $\rho$ and $\bar \rho$ and $k=k(\bar\rho, j)$. The dependence of $k$ on $\bar \rho$ and $j$ is inevitable.  For example,  if $\bar \rho=0$ and hence $\rho\equiv 0$, then $J_{k,\, j, \,\rho}=J_{k,\, j}$, which might not even be invertible. This is a subtle step and we need to choose the cutoff functions $\psi_j$ and $\rho$ carefully. 

Once we construct $J_{k,\, j, \,\rho}$, we use the technique in \cite{He17} to construct a unique almost complex structure $\tilde J_k$ using $J_{k,\, j,\, \rho}$ such that it is compatible with $g$. We assert that
\begin{equation}\label{estimate002}
\int_M |\nabla J_k-\nabla \tilde J_k|^4dv\leq C \epsilon_0
\end{equation}
which implies that $\tilde J_k\in \sigma$ using a theorem of B. White \cite[Theorem 2 and Section 6]{White89}. 

Here comes another essential point of the proof.
By our construction, $\tilde J_k$ equals $J_k$ outside $B_1$ and agrees with $J_0$ in $B_{1-j^{-1}}$. 
 If $\tilde J_k$ is  an energy-minimizing sequence in $\sigma$, this implies that $\l\equiv 0$ in $B_{1}$, which implies $\nu\equiv 0$ in $B_1$ with a bit more work. 
 Unfortunately we are not able to prove \begin{equation}\label{key000}
  \int_{B_1\setminus B_{1-j^{-1}}}|\Delta \tilde J_{k}|^2dv\rightarrow 0
 \end{equation}
 for $k\geq k_j$ and $j\rightarrow \infty$.
 Instead we prove the following inequality approximately (the actual estimate is much more complicated),
 \begin{equation}\label{key001}
 \int_{B_1\setminus B_{1-j^{-1}}}|\Delta \tilde J_{k}|^2dv\lessapprox C \int_{B_{1}\setminus B_{1-j^{-1}}} (|\nabla^2 J_k|^2+|\nabla J_k|^4)dv.
 \end{equation}
 Given \eqref{key001} and the fact that $\{J_k\}$ is an energy-minimizing sequence, we can obtain that
 \begin{equation}
\l(B_1)\leq C\nu (\p B_1). 
\end{equation} 
Since the defect measure $\nu(\p B_1)$ can be strictly positive on $\p B_1$, this does not direct lead to the conclusion $\l \equiv 0$. However the construction above actually works on any ball $B_r\subset B_2$, $r\in [1/4, 7/4]$ (replacing $B_1$ by $B_r$) and the arguments can be directly carried over. Hence we will prove, for $r\in [1/4, 7/4]$, 
 \begin{equation}
 \l(B_r)\leq C\nu (\p B_r). 
\end{equation} 
In particular we have for $r\in [3/2, 7/4]$,
\begin{equation}
\l(B_{3/2})\leq C\nu(\p B_r)
\end{equation}
Since $\nu$ is a totally bounded positive Radon measure, then $\nu(\p B_r)=0$ for infinitely many $r$ (actually $\nu(\p B_r)>0$ for at most countably many $r$). Hence this proves that $\l\equiv 0$ in $B_{3/2}$. It is then a standard practice to prove that $\nu\equiv 0$ in $B_1$. \\

Now we carry out the details. First we state two versions of the Poincare inequality we will use below. For $f\in W^{1, 2}(B_R), B_R\subset \R^n$,  denote $\underline{f}$ to be its average in the ball,
\[
\underline{f}=\frac{1}{\text{Vol}(B_R)}\int_{B_R} f(y)dy
\] 
Then we have
\begin{equation}\label{poincare01}
 R^{-n}\int_{B_R}|f-\underline{f}|^2dy\leq C R^{2-n}\int_{B_R}|Df|^2dy,
\end{equation}
where $C$ is a uniform dimensional constant. Suppose $\phi$ is a cutoff function supported in $B_R$ such that $\int_{B_R}\phi(y)dy=1$. We denote 
\[
f_{*}=\int_{B_R} \phi (y) f(y)dy,
\]
then we have
\begin{equation}\label{poincare02}
 R^{-n}\int_{B_R}|f-f_{*}|^2dy\leq C R^{2-n}\int_{B_R}|Df|^2dy.
\end{equation}
We should mention that $f$ can be taken as vector-valued and matrix-valued functions in \eqref{poincare01} and \eqref{poincare02}, as a direct generalization. \\

We say $J$ is $g$-compatible if $(J+g J^t g^{-1})(x)=0$ (correspondingly $g(J\cdot, \cdot)+g(\cdot, J\cdot)=0$) holds as a matrix-valued equation for all $x$, where $J^t$ is the transpose of $J$. 
This condition is slightly easier to deal with than the condition $g(J\cdot, J\cdot)=g(\cdot, \cdot)$. Note that $S=\frac{1}{2}(J+g J^t g^{-1})(x))$ represent the $g$-symmetric part of $J$ with respect to $g$ and it plays an important role in \cite{He17} when we construct various comparison almost complex structures, see \cite[Section 3]{He17} for more details. If $J$ is compatible with $g$, then $J_\rho$ is almost $g$-compatible if $\bar\rho$ is sufficiently small (we assume that $|J|$ is bounded of course).
We estimate
\begin{equation}\label{gcompatible01}
\begin{split}
|(J_\rho+gJ_\rho^t g^{-1})(x)|=&\left|\int_{B_\rho(x)}\phi_\rho(y-x) (J(y)+g(x)J^t(y)g^{-1}(x))dy\right|\\
\leq &\rho^{-4}\int_{B_\rho}\left|J(y)+g(x)J^t(y)g^{-1}(x)\right|dy\\
\leq &\rho^{-4}\int_{B_\rho} \left|-g(y)J^t(y)g^{-1}(y)+g(x)J^t(y)g^{-1}(x)\right|dy\\
\leq& C\delta_0 \bar \rho
\end{split}
\end{equation}
where we use the facts
\[
\begin{split}&g(x)J^t(y)g^{-1}(x)-g(y)J^t(y)g^{-1}(y)=(g(x)-g(y)) J^t(y)g^{-1}(x)+g(x)J^t(y)(g^{-1}(x)-g^{-1}(y))\\
&\mbox{and}\,\,|g(x)-g(y)|\leq C\delta_0\bar \rho\,\, \mbox{for}\,\, |y-x|\leq \rho.
\end{split}
\] 
It is clear that \eqref{gcompatible01} holds for all the cases $J=J_0,\, J_k\, \mbox{and}\, J_{k,\, j}$ . \\

Next we show that $J_{k,\, j,\,\rho}$ is ``almost" an almost complex structure in the sense that 
$|J_{k,\, j,\,\rho}J_{k,\,j,\,\rho}+\text{id}|$ is very small pointwise. 
Now we specify $\rho$ further and discuss the properties of $J_{k,\, j,\,\rho}$. For $|x|\in [0, 1-j^{-1}]\cup [1, 2]$, since $\rho(x)=0$ we have
\begin{equation}\label{cutoff01}
J_{k,\,j, \,\rho(x)}(x)=J_{k,\, j}(x)=\begin{cases} &J_k(x), |x|\in [1, 2],\\
&J_0(x), |x|\in [0, 1-j^{-1}].\end{cases}\end{equation}

Fix $\delta_1>0$ sufficiently small (which can be taken as $j^{-1}$). 
We write $(1-j^{-1}, 1)$ as three subintervals \[(1-j^{-1}, 1-j^{-1}+\delta_1 j^{-1}]\cup (1-j^{-1}+\delta_1 j^{-1}, 1-\delta_1 j^{-1})\cup [1-\delta_1 j^{-1}, 1).\]  The discussions in each subinterval are different. 
We choose $\rho$ such that
\begin{equation}\label{rho01}
\begin{cases}
&\rho(1-j^{-1}+\delta_1 j^{-1})=\rho(1-\delta_1 j^{-1})=\delta_1 \bar\rho\\
&\rho(|x|)<\delta_1\bar \rho, |x|\in (1-j^{-1}, 1-j^{-1}+\delta_1 j^{-1})\cup (1-\delta_1 j^{-1}, 1)\\
&\rho(|x|)\geq \delta_1\bar \rho, |x|\in (1-j^{-1}+\delta_1 j^{-1}, 1-\delta_1 j^{-1}).
\end{cases}
\end{equation}
For $|x|\in (1-j^{-1}, 1-j^{-1}+\delta_1 j^{-1}]$, we have \[|x+\rho(x)z|\leq 1-j^{-1}+\delta_1 j^{-1}+\delta_1\bar \rho.\]
Hence $\psi_j(|x+\rho(x)z|)\leq 3j(\delta_1 j^{-1}+\delta_1\bar \rho)\leq 4\delta_1$ (we use $\psi_j(1-j^{-1})=0$ and $|\psi^{'}_j|\leq 3j$).
We compute
\begin{equation}\label{average001}
|J_{k,\, j,\,\rho(x)}(x)-J_{0,\, \rho(x)}(x)|=\left|\int_B\phi(z)\psi_j(|x+\rho(x)z)|(J_0-J_k)(x+\rho(x)z)dz\right|\leq 100 \delta_1. 
\end{equation}
Similarly, for $|x|\in [1-\delta_1 j^{-1}, 1)$, we have
\begin{equation}\label{average002}
|J_{k,\, j,\,\rho(x)}(x)-J_{k,\, \rho(x)}(x)|\leq 100\delta_1. 
\end{equation}
We can estimate
\begin{equation}\label{almost001}
|J_{0,\,\rho(x)}(x)-J_0(x)|\leq \int_B \phi(z)|J_0(x+\rho(x)z)-J_0(x)|dz\leq C_1 \bar\rho,
\end{equation}
where $C_1=C_1(\max|\nabla J_0|)$ is a uniform constant. 
Note that \eqref{almost001} holds for all $x\in B_1$.
Using \eqref{almost001} and \eqref{average001}, we get that
\begin{equation}\label{almost002}
|J_{k,\, j,\,\rho(x)}-J_0(x)|\leq 100\delta_1+C_1 \bar \rho, |x|\in (1-j^{-1}, 1-j^{-1}+\delta_1 j^{-1}]
\end{equation}
Hence \eqref{almost002} implies that $J_{k,\, j,\,\rho(x)}$ is close to an almost complex structure for $|x|\in (1-j^{-1}, 1-j^{-1}+\delta_1 j^{-1})$, provided that $100\delta_1+C_1 \bar \rho$ is sufficiently small. 
Since $|\nabla J_k|$ might not be uniformly bounded, we do not have an effective pointwise estimate on $|J_{k,\,\rho(x)}(x)-J_k(x)|$ as above. Instead we apply the Poincare inequality \eqref{poincare02} in the ball $B_\rho(x)$ with
\[
J_{*}=J_\rho (x)=\int_{B_\rho(x)} \phi_\rho(y-x) J(y)dy,
\]
and hence we have
\begin{equation}\label{poincare03}
\rho^{-4}\int_{B_{\rho}(x)} |J(y)-J_\rho(x)|^2 dv\leq C\rho^{-2}\int_{B_\rho(x)}|DJ(y)|^2dy,
\end{equation}
where we  choose $\rho=\rho(x)>0$ and the Poincare inequality is applied to the matrix valued function $J(y)$ in the ball $B_{\rho(x)}(x)$. Replacing $J$ by $J_k$, we can get that
\begin{equation}\label{poincare004}
\rho^{-4}\int_{B_\rho(x)}|J_k(y)-J_{k,\,\rho(x)}(x)|^2dy\leq C\rho^{-2}\int_{B_\rho(x)}|DJ_k(y)|^2dy,
\end{equation}
By H\"older inequality, we see that
\[
\rho^{-2}\int_{B_\rho(x)}|DJ_k(y)|^2dy\leq C\left(\int_{B_\rho(x)}|DJ_k|^4dy\right)^{\frac{1}{2}}
\]
It follows that
\begin{equation}\label{poincare04}
\rho^{-4}\int_{B_\rho(x)}|J_k(y)-J_{k,\,\rho(x)}(x)|^2dy\leq C\sqrt{\epsilon_0}
\end{equation}
Hence \eqref{poincare04} implies that there are many $y$s in $B_{\rho}(x)$ such that $J_k(y)$ is close to $J_{k,\,\rho(x)}(x)$ in the sense that
\[
|J_k(y)-J_{k,\, \rho(x)}(x)|\leq C\sqrt[4]{\epsilon_0}.
\]
In particular this implies that
\[
|J_{k,\, \rho(x)}(x)J_{k,\, \rho(x)}(x)+\text{id}|\leq C\sqrt[4]{\epsilon_0}
\]
Using \eqref{average002} and the above, we get that
\begin{equation}\label{almost003}
|J_{k,\, j,\,\rho(x)} J_{k,\, j,\,\rho(x)}+\text{id}|\leq C(\delta_1+\sqrt[4]{\epsilon_0}), |x|\in [1-\delta_1 j^{-1}, 1). 
\end{equation}
Next we consider $|x|\in (1-j^{-1}+\delta_1j^{-1}, 1-\delta_1 j^{-1})$, where $\rho(x)\geq \delta_1\bar \rho$. 
We compute
\begin{equation}\label{average003}
\begin{split}
&|J_{k,\, j,\,\rho(x)}(x)-J_{0,\, \rho(x)}(x)|\leq \int_B \phi(z) |(J_0-J_k)(x+\rho(x)z)|dz
\end{split}
\end{equation}
We have
\[
\begin{split}
\int_B \phi(z) |(J_0-J_k)(x+\rho(x)z)|dz\leq &{\rho(x)}^{-4} \int_{B_{\rho(x)}(x)}|J_0(y)-J_k(y)|dy\\
\leq&C\rho^{-2}\|J_0-J_k\|_{L^2(B_{3/2})}
\end{split}
\]
Since $J_k$ converges to $J_0$ strongly in $W^{1, 2}$ and $\rho\geq \delta_1\bar \rho$ for $|x|\in (1-j^{-1}+\delta_1j^{-1}, 1-\delta_1 j^{-1})$, 
we can choose $k_0=k_0(\delta_1, \bar\rho)$ sufficiently large such that 
\[
C\rho^{-2} \|J_0-J_k\|_{L^2}\leq \bar\rho
\]
Hence we get, for $k\geq k_0$, 
\begin{equation}\label{poincare05}
{\rho}^{-4} \int_{B_{\rho(x)}(x)}|J_0(y)-J_k(y)|dy\leq \bar\rho
\end{equation}
Using \eqref{poincare05} and \eqref{average003} we have that, for $|x| \in (1-j^{-1}+\delta_1j^{-1}, 1-\delta_1 j^{-1})$ and $k\geq k_0$,
\[
|J_{k,\, j,\,\rho(x)}(x)-J_{0,\, \rho(x)}(x)|\leq \bar \rho.
\]
This together with \eqref{almost001} implies that, for  $|x| \in (1-j^{-1}+\delta_1j^{-1}, 1-\delta_1 j^{-1})$ and $k\geq k_0$,
\[
|J_{k,\, j,\,\rho(x)}(x)-J_0(x)|\leq C\bar\rho,
\]
and in particular we have. 
\[
|J_{k,\, j,\,\rho(x)}(x)J_{k,\, j\,, \rho(x)}(x)+\text{id}|\leq C \bar \rho. 
\]
In the discussion above, we take $\delta_1=j^{-1}$ and $\bar\rho=(10j^2)^{-1}$ and hence $k_0=k(j)=k_j$. \\

We summarize the ideas briefly of our construction $J_{k,\, j,\, \rho}$ in $(1-j^{-1}, 1)$ above.  Note that $J_{k,\, j,\, \rho}$ is the local average of $J_{k,\, j}=J_k+(J_0-J_k)(1-\psi_j)$ and the average radius $\rho$ depends on the location $|x|$. Near the boundary  $|x|=1-j^{-1}$ and $|x|=1$ of the annulus $B_1\setminus B_{1-j^{-1}}$, $\rho$ is very small hence the average effect does not prevail. But in this tiny region, the cutoff function $\psi$ is either very close to $0$ or $1$ (pointwise) and hence $J_{k,\, j}$ is either close to $J_0$ (near $|x|\in (1-j^{-1}, 1-j^{-1}+\delta_1 j^{-1}]$) or $J_k$ (for  $|x|\in [1-\delta_1 j^{-1}, 1)$). It follows that $J_{k,\, j,\, \rho}$ is either close to $J_{0,\, \rho}$ or $J_{k, \,\rho}$, which are both close to an almost complex structure given $\rho$ sufficiently small. Note that in this step the smallness of energy plays a very important role. 
For $|x|\in (1-j^{-1}+\delta_1 j^{-1}, 1-\delta_1 j^{-1})$, which is at least $\delta_1 j^{-1}$ away from the boundary and $\rho\geq \delta_1 \bar \rho$, the average effect prevails  in this region since $\rho$ is "relatively" large. Since $J_k$ converges to $J_0$ strongly in $W^{1, 2}$, which can be understood as in average $J_k$ is approximately $J_0$ for $k$ sufficiently large,  and hence $J_{k,\, j,\, \rho(x)}(x)$ is close to $J_0(x)$ pointwise by choosing $k$ sufficiently large. Simply put, the interplay of the cutoff function $\psi_j$ and the function $\rho$ is the key for our construction. \\

Next we construct a sequence $\tilde J_k(x)$ using $J_{k,\, j, \,\rho(x)}(x)$ by the technique we have used in \cite[see (4.15), (4.16)]{He17}. We briefly recall the construction. 
In the following we consider $k\geq k_j$ for each $j$. 
For $|x|\in [0, 1-j^{-1}]\cup [1, 2]$, $\rho=0$ and hence \begin{equation}\label{cutoff001}\tilde J_k=J_{k,\, j, \, \rho}=J_{k,\, j}=
\begin{cases}
J_0,\, |x|\leq 1-j^{-1}\\
J_k,\, |x|\geq 1
\end{cases}\end{equation} We only need to consider $|x|\in (1-j^{-1}, 1)$. 
Let $S_{g}(x)$ and $A_{g}(x)$  be the $g$-symmetric and $g$-skew symmetric part of $J_{k,\, j, \,\rho(x)}(x)$ respectively. We have by \eqref{gcompatible01}
\[
|S_{g}(x)|=\frac{1}{2} |(J_{k,\,j, \,\rho(x)}(x)+g(x)J^t_{k,\,j, \,\rho(x)}(x)g^{-1}(x))|\leq C\delta_0\bar \rho
\]
 It follows that, 
\[|A_{g}^2+\text{id}|<C\sqrt[4]{\epsilon_0}+C\delta_0 {\bar\rho}.\]
Hence $-A_{g}^2$ is $g$-symmetric (since $A_{g}$ is $g$-skew symmetric) and it is close to the identity matrix pointwise, in particular it is invertible. Denote  $Q_{g}$ to be the $g$-symmetric matrix such that $Q_{g}^2=-A_{g}^2$. Note that  $Q_{g}$ is uniquely determined and it commutes with $A_g$ (viewed as a matrix-valued function, it is the unique square root of the positive definite symmetric matrix $-A_{g}^2$ at each point). Denote $\tilde J_{k}(x)=Q_{g}^{-1}(x)A_{g}(x)$ for $x\in B_2$, then $\tilde J_{k}(x)$ is a $g$-compatible almost complex structure in $B_2$.
We extend $\tilde J_k$ to $M$ by simply putting $\tilde J_k=J_k$  on $M\setminus B_{1}(p)$. 
Now we establish \eqref{estimate002}.  We only need to consider over $B_1(p)$. 
Note that the $L^\infty$ norm of $J_{k,\, j}, Q_g, A_g, Q_g^{-1}$ and $1/|Q_{g}^{-1}|$ are all uniformly bounded by a dimensional constant. 
Since $Q_g^2=-A_g^2$ and $|A_g^2+\text{id}|\ll 1$, we have the following expansion of matrix,
\begin{equation}\label{matrix001}
Q_g=\sqrt{\text{id}-(\text{id}+A_g^2)}=\sum_{l=0}^\infty \binom{1/2}{l}(\text{id}+A_g^2)^l.
\end{equation}
We can compute directly that $|\nabla Q_g|\leq C|A_g||\nabla A_g|\leq C|\nabla A_g|$. Hence we obtain, 
\[
|\nabla \tilde J_k|\leq C|\nabla A_g|
\]
We also need (using $\nabla g=0$),
\[
|\nabla S_g|=|\nabla J_{k, j, \rho}+g \nabla J^{t}_{k, j, \rho} g^{-1}|\leq C |\nabla J_{k, j, \rho}|.
\]
We compute
\[
|\nabla J_{k, j, \rho}|\leq C\int_B \phi(z) |\nabla J_{k, j}(x+\rho(x)z)| (1+|\rho^{'}|)dz
\]
Hence we have
\[
|\nabla J_{k, j, \rho}|^4\leq C \int_B  |\nabla J_{k, j}(x+\rho(x)z)|^4dz
\]
Moreover, we have $|\nabla A_g|\leq |\nabla J_{k, j, \rho}|+|\nabla S_g|\leq C|\nabla J_{k, j, \rho}|$. It follows that
\[
\int_B |\nabla \tilde J_k|^4dv\leq C\int_B \left(\int_B  |\nabla J_{k, j}(x+\rho(x)z)|^4dz\right)dv_x\leq C \int_{B_{3/2}} |\nabla J_{k, j}|^4dv
\]
Since $\nabla J_{k, j}=\nabla J_k+(1-\psi_j)(\nabla J_0-\nabla J_k)-(J_0-J_k) \nabla \psi_j$ and $|\nabla \psi_j|\leq 3j$, it follows that
\begin{equation}
\int_{B_{3/2}} |\nabla J_{k, j}|^4dv\leq C \int_{B_2} (|\nabla J_k|^4+|\nabla J_0|^4) dv+C j^4\int_B |J_0-J_k|^4dv
\end{equation}
Using the Sobolev inequality, we know that
\[
|J_0-J_k|_{L^4}\leq C|J_0-J_k|_{W^{1, 2}}
\]
By choosing $k_j=k(j)$ sufficiently large such that for $k\geq k_j$,we can assume that
\begin{equation}\label{j001}
Cj^4\int_{B_1}(|J_0-J_k|^2+|J_0-J_k|^4+|\nabla J_0-\nabla J_k|^2)dv \leq j^{-1}\leq \epsilon_0.
\end{equation}
This establishes \eqref{estimate002}.  Then we can quote a theorem of B. White \cite[Theorem 2 and Section 6]{White89} to assert that $\tilde J_k$ and $J_k$ are in the same homotopy class, for $k\geq k_j$.\\

Fix $\epsilon>0$. Since $J_k$ is an energy-minimizing sequence, for $k$ sufficiently large we have, 
\[
\int_M |\Delta J_k|^2dv\leq \int_M |\Delta \tilde J_k|^2 dv+\epsilon.
\]
By the construction of  $\tilde J_k$, we get
\begin{equation*}
\int_{B_1}|\Delta J_k|^2 dv\leq \int_{B_{1}\setminus B_{1-j^{-1}}}|\Delta \tilde J_k|^2dv+\int_{B_{1-j^{-1}}} |\Delta J_0|^2 dv+\epsilon
\end{equation*}
By taking $j\rightarrow \infty$ (hence $k\geq k_j\rightarrow \infty$), we get (since $B_1$ is open)
\begin{equation*}
\l(B_1)+\int_{B_{1}}|\Delta J_0|^2 dv\leq \liminf \int_{B_{1}\setminus B_{1-j^{-1}}}|\Delta \tilde J_k|^2dv+\int_{B_{1-j^{-1}}} |\Delta J_0|^2 dv+\epsilon
\end{equation*}
Hence we get
\begin{equation*}
\l(B_1)\leq  \liminf_{j\rightarrow \infty} \int_{B_{1}\setminus B_{1-j^{-1}}}|\Delta \tilde J_k|^2dv+\epsilon. 
\end{equation*}
Since $\epsilon>0$ is arbitrary, we have established the following estimate,
\begin{equation}\label{key002}
\l(B_1)\leq  \liminf_{j\rightarrow \infty} \int_{B_{1}\setminus B_{1-j^{-1}}}|\Delta \tilde J_k|^2dv.  
\end{equation}
Now we need estimates as in \eqref{key001} to control the righthand side of \eqref{key002}. Recall that we have the unique decomposition $J_{k,\, j, \,\rho}=A_g+S_g$ and $\tilde J_k=Q_g^{-1}A_g$, where $Q_g$ is the unique square root of $-A_g^2$.
Using \eqref{matrix001} we have
\[
|\nabla Q_g|\leq C|\nabla A_g|, |\Delta Q_g|\leq C(|\Delta A_g|+|\nabla A_g|^2)
\]
It follows that
\[
|\Delta (Q^{-1}_g A_g)|\leq C(|\Delta A_g|+|\nabla A_g|^2)\leq C(|\Delta J_{k, \, j, \,\rho}|+|\nabla J_{k,\, j, \,\rho}|^2)
\]
Hence we obtain 
\begin{equation}\label{estimate400}
\int_{B_{1}\setminus B_{1-j^{-1}}}|\Delta \tilde J_k|^2dv\leq C\int_{B_{1}\setminus B_{1-j^{-1}}}\left( |\Delta J_{k, \, j, \,\rho}|^2+|\nabla J_{k,\, j, \,\rho}|^4\right)dv.
\end{equation}
We compute
\[
\begin{split}
|\nabla J_{k,\,j}|=&|\nabla J_k+(1-\psi_j)(\nabla J_0-\nabla J_k)-\nabla \psi_j (J_0-J_k)|\\
=&|\psi_j \nabla J_k+(1-\psi_j)\nabla J_0-\nabla \psi_j (J_0-J_k)|\\
\leq& |\nabla J_k|+|\nabla J_0|+Cj |J_0-J_k|\end{split}
\]
Similarly we compute
\[
\begin{split}
|\Delta J_{k,\,j}|=&|\Delta J_k+\Delta [(1-\psi_j)( J_0-J_k)]|\\
\leq& |\nabla^2 J_k|+|\nabla^2 J_0|+Cj^2 (|J_0-J_k|+|\nabla J_0-\nabla J_k|)
\end{split}
\]
Write $y=x+\rho(x)z$, then we have
\[
\left|\frac{\p y_i}{\p x_j}\right|\leq C(1+|\rho^{'}|), \left|\frac{\p^2 y_i}{\p x_j \p x_k}\right|\leq C(1+|\rho^{'}|+|\rho^{''}|). 
\]
Using $|\rho^{'}|+|\rho^{''}|\leq 1$, $|\psi^{'}_j|+|\psi^{''}_j|\leq 20j^2$, we can then get
\[
\begin{split}
|\nabla J_{k,\, j, \,\rho}|=&\left|\int_B \phi(z) \nabla_x J_{k,\, j}(x+\rho(x)z)dz\right|\\
\leq &C\int_B \phi(z)|\nabla_z J_{k,\, j}| (1+|\rho^{'}|)dz\\
\leq&  C\int_{B_1} \phi(z)\left(|\nabla J_k|+|\nabla J_0|+j |J_0-J_k|\right)dz,
\end{split}
\]
where the function is evaluated at $y=x+\rho(x)z$.  For any open set $U\subset B_2$, denote  $U_{\bar\rho}=\{x: \text{dist}(x, U)<\bar\rho\}$. We have
\begin{equation}\label{estimate401}
\begin{split}
\int_{U}|\nabla J_{k,\, j, \,\rho}|^4dv_x\leq& C\int_{U}\left(\int_{B_1} \phi(z)\left(|\nabla J_k|(y)+|\nabla J_0|(y)+j |J_0-J_k|(y)\right)dz\right)^4dv_x\\
\leq &C\int_{U_{\bar\rho}} (|\nabla J_k|^4+|\nabla J_0|^4+j^4 |J_0-J_k|^4)dv,
\end{split}
\end{equation}
where we have used  a standard technique to estimate the $L^p$ norm of mollifier approximation in the last step (see \cite[Lemma 7.2, (7.15)]{GT} for example).
Similarly we compute
\[
\begin{split}
|\Delta J_{k,\, j, \,\rho}|=&\left|\int_B \phi(z) \Delta_x \left(J_{k,\, j}(x+\rho(x)z)\right)dz\right|\\
\leq &C\int_B \phi(z)\left(|\nabla^2 J_k|+|\nabla^2 J_0|+j^2(|\nabla J_0-\nabla J_k|+|J_0-J_k|)\right) dz
\end{split}
\]
It follows that
\begin{equation}
\int_{U}|\Delta J_{k,\, j, \,\rho}|^2dv_x\leq C \int_{U_{\bar \rho}} \left(|\nabla^2 J_k|^2+|\nabla^2 J_0|^2+j^4|J_0-J_k|^2+j^4|\nabla J_0-\nabla J_k|^2\right)dv_x
\end{equation}
Take $U=B_1\setminus B_{1-j^{-1}}$ in the above, we obtain, 
\begin{equation}\label{estimate403}
\int_{U} (|\nabla J_{k,\, j, \,\rho}|^4+|\nabla^2 J_{k,\, j, \,\rho}|^2)dv\leq C\int_{U_{\bar \rho}} (|\nabla^2 J_k|^2+|\nabla J_k|^4)dv+C(R_1+R_2),
\end{equation}
where the remainder terms read,
\[
\begin{split}
R_1=&\int_{U_{\bar \rho}} (|\nabla J_0|^4+|\Delta J_0|^2 )dv\\
R_2=&j^4\int_{U_{\bar \rho}} (|J_0-J_k|^2+|J_0-J_k|^4+|\nabla J_0-\nabla J_k|^2)dv
\end{split}
\]
Recall we assume that $k\geq k_j$ such that \eqref{j001} holds, hence $R_2\leq j^{-1}$. Certainly when $j\rightarrow \infty$, the Lebesgue measure $U_{\bar \rho}\rightarrow 0$ and hence $R_1\rightarrow 0$. By \eqref{estimate403}, we obtain
\begin{equation}\label{estimate404}
\liminf_{j\rightarrow \infty}\int_{B_1\setminus B_{1-j^{-1}}}(|\nabla J_{k,\, j, \,\rho}|^4+|\Delta J_{k,\, j, \,\rho}|^2)dv\leq 
C\liminf_{j\rightarrow\infty}\int_{U_{\bar \rho}} (|\nabla^2 J_k|^2+|\nabla J_k|^4)dv
\end{equation}
Hence we obtain, by \eqref{estimate400} and \eqref{estimate404},
\begin{equation}\label{estimate405}
\liminf_{j\rightarrow \infty}\int_{B_{1}\setminus B_{1-j^{-1}}}|\Delta \tilde J_k|^2dv\leq C\liminf_{j\rightarrow \infty}\int_{U_{\bar\rho}} (|\nabla^2J_k|^2+|\nabla J_k|^4)dv
\end{equation}
Note that $U=B_{1}\setminus B_{1-j^{-1}}$, $\bar \rho=j^{-2}/10$. Fix $\epsilon>0$. Denote
\[
B_{1,\,\epsilon}=\{x: \text{dist}(x, \p B_1)<\epsilon\}. 
\]
For any $j$ sufficiently large, $U_{\bar \rho}\subset \overline{B_{1,\,\epsilon}}$. Hence we have
\[
\lim_{k\rightarrow\infty}\int_{\overline{B_{1,\,\epsilon}}} (|\nabla^2J_k|^2+|\nabla J_k|^4)dv=\lim_{k\rightarrow\infty}\mu_k(\overline{B_{1,\,\epsilon}})\leq \mu(\overline{B_{1,\,\epsilon}}),
\]
where  we use the fact that $\mu_k$ converges to $\mu$ weakly and $\overline{B_{1,\,\epsilon}}$ is a closed set. Together with \eqref{estimate405}, this gives that
\begin{equation}\label{key003}
\lim_{j\rightarrow \infty}\int_{B_{1}\setminus B_{1-j^{-1}}}|\Delta \tilde J_k|^2dv\leq C\mu(\overline{B_{1,\,\epsilon}}). 
\end{equation}
Note that $\cap_{\epsilon>0} \overline{B_{1,\,\epsilon}}=\p B_1$,  we obtain 
$\lim_{\epsilon\rightarrow 0}\mu(\overline{B_{1,\,\epsilon}})=\nu(\p B_1).$
With \eqref{key003}, we obtain
\begin{equation}\label{key004}
\lim_{j\rightarrow \infty}\int_{B_{1}\setminus B_{1-j^{-1}}}|\Delta \tilde J_k|^2dv\leq C\nu(\p B_1). 
\end{equation}
Hence we have obtained the desired estimate
\begin{equation}\label{key005}
\l(B_1)\leq C \nu(\p B_1). 
\end{equation}
In the proof of \eqref{key005}, we work in the ball $B_1\subset B_2$ and do the constructions in $B_{1}\setminus B_{1-j^{-1}}$. We replace $B_1$ by $B_r$, for $r\in [1/4, 7/4]$ and replace $B_{1}\setminus B_{1-j^{-1}}$ by 
$B_{r}\setminus B_{r-j^{-1}}$ and we apply the same arguments to obtain, 
\begin{equation}\label{key006}
\l(B_r)\leq C\nu(\p B_r). 
\end{equation}
This completes the proof.
\end{proof}

\section{Appendix}
\subsection{Harmonic and biharmonic almost symplectic structure}
Given an almost Hermitian metric $(M, g)$ with compatible almost complex structures $\cJ_g$, we can ask whether there is a compatible almost complex structure, such that $(g, J)$ defines an almost K\"ahler structure (that is $d\omega=0$.)
We can ask to minimize the following the energy functional $\tilde E_1$,
\begin{equation}\label{hs1}
\tilde E_{1}(\omega)=\int_M (|d\omega|^2+|d^*\omega|^2)dv_g, \omega=g(J\cdot, \cdot)
\end{equation}
and the energy functional $\tilde E_2$,
\begin{equation}\label{hs2}
\tilde E_2(\omega)=\int_M |\Delta_d\omega|^2 dv_g, \omega=g(J\cdot, \cdot)
\end{equation}                      
We use the notations as in Hodge theory, $d^*$ is the formal adjoint of $d$, and $\Delta_d=dd^*+d^*d$. 
Note that zero energy of $\tilde E_1$ and $\tilde E_2$ corresponds to a symplectic structure which is compatible with $g$ (hence $(M, g)$ is an almost K\"ahler structure).

Minimizing the energy functionals $\tilde E_1$ and $\tilde E_2$ is parallel to theory of harmonic almost complex structure and biharmonic almost complex structure, even though the zero energy corresponds to different geometric structures. It is relatively easy to see that the energy-minimizers have very similar regularity theory. For simplicity, we focus on $\tilde E_2(\omega)$ on a compact almost Hermitian manifold $(M, g)$ and its critical point will be called \emph{biharmonic almost symplectic structure}. By Weitzenbock formula,  we have
\[
\Delta_d \omega=\Delta \omega+Rm*\omega
\]                               
Hence minimizing $\tilde \cE_2(\omega)$  is equivalent to minimizing 
\[
\int_M |\Delta \omega+Rm*\omega|^2 dv_g
\]                     
Note that $\omega=g(J\cdot, \cdot),$ hence we can rewrite the energy functional in terms of $J$, 
\[
\int_M |\Delta J+Rm*J|^2dv_g
\]
This functional differs $\cE_2(J)=\int_M |\Delta J|^2 dv_g$ only be lower order terms of the form
\[
\int_M (\Delta J, Rm*J)dv_g+\int_M |Rm*J|^2 dv_g
\]       
Hence, the Euler-Lagrangian equation differs by lower order linear terms,
\[
\Delta^2 J=Q(J, \nabla J, \nabla^2 J, \nabla^3 J)+L(J, \nabla J, \Delta J),
\]                                      
where $L$ takes of the form
\[
L(J, \Delta J)=\Delta J*Rm+\nabla Rm*\nabla J+a(Rm, \nabla Rm, \Delta Rm) *J.
\]
Regularity theory in \cite{HJ19} applies directly to get smoothness of weakly $W^{2, 2}$-biharmonic almost symplectic structure. Moreover, arguments as in Theorem \ref{thm-existence} hold also for $\tilde \cE_2(\omega)$, to get an energy-minimizing biharmonic almost symplectic structure. And hence such objects form a compact set in smooth topology, parallel to biharmonic almost complex structures. Moreover, since the principle operator is the same, we know that the index of its linearized operator is zero. Hence in general we expect that the moduli (of energy-minimizing biharmonic almost symplectic structures) consists of finitely many points. We summarize the discussions above as follows,

\begin{thm}A $W^{2, 2}$-biharmonic almost symplectic structure on any compact almost Hermitian manifold $(M^4, g)$ is smooth. There always exist minimizers of the energy functional $\tilde \cE_2(\omega)$ which are smooth biharmonic almost symplectic structures on any compact almost Hermitian manifold $(M^4, g)$ and all such energy-minimizers form a compact set. 
\end{thm}

We can also consider the energy functional $\tilde E_2(\omega)$ for almost complex structures in a fixed homotopy class, to discuss biharmonic almost symplectic structures in a fixed homotopy class (of almost complex structures). Similar results as in Theorem \ref{main2} and Theorem \ref{main3} hold with only slight modifications. We believe these results shall have applications in symplectic geometry and we shall discuss these objects elsewhere.  

\subsection{Moduli and a conjecture}

A biharmonic almost complex structure with zero energy defines a K\"ahler structure and this makes 
a K\"ahler structure  a very special example of energy-minimizing biharmonic almost complex structures. A K\"ahler structure imposes a very strong topology restrictions on the topology of the underlying four manifold (indeed there are rather satisfactory classification of complex complex surfaces). Unlike a K\"ahler structure, an energy-minimizing biharmonic almost structures always exists on a compact four manifold which supports an almost complex structure.  

We regard the Riemannian metric as a parameter and consider the moduli space of (energy-minimizing) biharmonic almost complex structure.  
The moduli of biharmonic almost complex structures should have important applications for the study of smooth topology of almost complex four manifolds.                                        
For a general theory of moduli problem, see \cite[Section 4.2]{DK}. We briefly discuss a conjecture without diving into the technical terminology.

We have known that energy-minimizing biharmonic almost complex structures form a compact set (in smooth topology). In particular, there are at most finitely many homotopy classes which can have an absolute energy-minimizer. 
One can easily see that the linearized operator is bi-Laplacian operator plus lower order term, which is  a \emph{Fredholm} operator. In particular, the principle operator is self-adjoint and hence the index of the linearized operator is zero. 
Hence for a ``generic" metric, there are finitely many energy-minimizers and they form a smooth manifold. 
On the other hand the uniqueness fails. For example, let $g$ be a Calabi-Yau metric on a $K3$ surface, which supports a hyperk\"ahler structure. For such a metric $g$, it has a $S^2$ family of compatible complex structures 
and $g$ is K\"ahler with respect to all these complex structures. 
Note that the failure of uniqueness makes the perturbation of biharmonic almost complex structures with varying metric a very tricky problem. 

A standard principle is then to consider the transversality. Suppose $(g, J)$ defines a compatible almost Hermitian structure. Denote $(h, K)$ to be the infinitesimal variation of $(g, J)$.
Then we have the following,
\begin{prop}The infinitesimal variation $(h, K)$ satisfies 
\begin{equation}\label{infi01}
JK+KJ=0;\; h(J\cdot, \cdot)+h(\cdot, J\cdot)+g(K\cdot, \cdot)+g(\cdot, K\cdot)=0.
\end{equation}
\end{prop}
\begin{proof}This is straightforward by the restrictions $J^2=-\text{id}$ and $g(J\cdot, \cdot)+g(\cdot, J\cdot)=0$. 
\end{proof}

We consider the map 
\[
B(g, J)=\left[J,\,  \Delta^2_g J\right]
\]
Then $J$ is a biharmonic almost complex structure with respect to the metric $g$ if $(g, J)$ is a zero point of $B$. We expect the following transversality result holds
\begin{conj}\label{conj01}
The map $B$ is transverse to the zero section and hence the zero set of $B$ is regular. 
\end{conj}

Suppose this were the case, then a rather standard discussion of moduli applies. In particular, then for a generic metric, the homotopy classes of an energy-minimizing biharmonic almost complex structure does not depend on the background metric, and hence gives a diffeomorphic invariant. Since for a generic K\"ahler metric $g$, there exists essentially one compatible complex structure (which defines a K\"ahler structure together with the metric), one might expect that such a homotopy class on an almost complex four manifold which contains an energy-minimizing biharmonic almost complex structures is unique. 

\begin{conj}\label{conj02}Given a compact (simply-connected, nonspin) four dimensional almost complex manifold $M$, the homotopy class of energy-minimizing biharmonic almost complex structures is a differential invariant of the underlying almost complex manifold and it is unique (up to the sign).  
\end{conj} 

If the conjecture were true,  it would imply that, for a given almost complex four manifold, there exists an \emph{a priori} homotopy class which supports an energy-minimizing biharmonic almost complex structures for a generic metric; in particular it would imply that the first Chern class of a simply-connected  K\"ahler surface is a differential invariant.  The above discussion is rather superfluous and we believe the further study of biharmonic almost complex structure in dimension four should lead to many applications. One main technical difficulty is certainly the transversality Conjecture \ref{conj01}. To prove such a result, one needs to prove that the linearized operator of $B$ is a surjective operator.  The main technical obstacle is that the linearized operator with respect to the linearization of the metric $g$ is an extremely complicated fourth order operator.
Due to its technical complexity, we shall explore the further study elsewhere.

\end{document}